\documentclass[12pt,reqno]{amsart}
\usepackage{amsmath, amsthm, amssymb, stmaryrd}
\usepackage{hyperref}
\usepackage{enumitem}
\usepackage{url}
\usepackage{tikz-cd}

\let\OLDthebibliography\thebibliography
\renewcommand\thebibliography[1]{
  \OLDthebibliography{#1}
  \setlength{\parskip}{0pt}
  \setlength{\itemsep}{0pt plus 0.3ex}
}

\usepackage{mathtools}

\usepackage{multirow, array}
\usepackage{placeins}
\usepackage{caption} 

\newtheorem{thm}{Theorem}[section]
\newtheorem{lemma}[thm]{Lemma}

\newtheorem{cor}[thm]{Corollary}

\theoremstyle{definition}
\newtheorem{defn}[thm]{Definition}

\theoremstyle{remark}

\numberwithin{equation}{section}

\newcommand*\wrapletters[1]{\wr@pletters#1\@nil}
\def\wr@pletters#1#2\@nil{#1\allowbreak\if&#2&\else\wr@pletters#2\@nil\fi}

 \def\Del{{\Delta}}

\def\eps{\varepsilon}

\def\le{\leqslant} \def\ge{\geqslant}

\def \bbA {\mathbb A}

\def \bbF {\mathbb F}

\def \bN {\mathbb N}

\def \bR {\mathbb R}
\def \bZ {\mathbb Z}

\def \bx {\mathbf x}

\def \by {\mathbf y}

\def \bzero {\mathbf 0}

\def \fl {\mathfrak l}

\def \fp {\mathfrak p}
\def \fq {\mathfrak q}

\def \fL {\mathfrak L}

\def \cG {\mathcal G}

\def \cL {\mathcal L}

\def \cP {\mathcal P}
\def \cQ {\mathcal Q}

\def \cX {\mathcal X}

\def \deg {\mathrm{deg}}

\begin{document}
\title[Incidence geometry and polynomial expansion]
{Incidence geometry and polynomial expansion over finite fields}
\author[Nuno Arala \and Sam Chow]{Nuno Arala \and Sam Chow}
\address{Mathematics Institute, Zeeman Building, University of Warwick, Coventry CV4 7AL}
\email{Nuno.Arala-Santos@warwick.ac.uk}
\address{Mathematics Institute, Zeeman Building, University of Warwick, Coventry CV4 7AL}
\email{Sam.Chow@warwick.ac.uk}
\thanks{}
\date{}

\begin{abstract} 
We use spectral theory and algebraic geometry to establish a higher-degree analogue of a Szemer\'edi--Trotter-type theorem over finite fields, with an application to polynomial expansion.
\end{abstract}

\maketitle

\section{Introduction}

\subsection{Incidence geometry}

The Szemer\'edi--Trotter theorem asserts that if $\cP$ is a set of points and $\cL$ is a set of lines in $\bR^2$ then
\[
\# \{ (\fp, \fl) \in \cP \times \cL: \fp \in \fL \}
\ll |\cP|^{2/3} |\cL|^{2/3} + |\cP| + |\cL|.
\]
Such estimates have implications for the sum--product problem of Erd\H{o}s and Szemer\'edi, see \cite{Ele1997}. Vinh  established the following finite-field analogue.

\begin{thm} [\cite{Vinh1}] Let $\cP$ be a set of points and $\cL$ a set of lines in $\bbF_q^2$. Then
\[
\left| 
\# \{ (\fp, \fl) \in \cP \times \cL: \fp \in \fL \}
- \frac{|\cP| \cdot |\cL|}{q} \right|
\le
q^{1/2} \sqrt{|\cP| \cdot |\cL|}.
\]
\end{thm}

In our previous paper, we bootstrapped Vinh's theorem to infer the following generalisation to higher degrees.

\begin{thm} [\cite{AC}] Let $n \in \bN$, and let $q > n$ be a prime power. Let $\cP$ be a set of points in $\bbF_q^2$, and $\cQ$ a collection of subsets of $\bbF_q^2$ of the shape
\[
\{ (x,y) \in \bbF_q^2:
y = a_n x^n + \cdots + a_0 \}.
\]
Then
\[
\left| 
\# \{ (\fp, \fl) \in \cP \times \cL: \fp \in \fL \}
- \frac{|\cP| \cdot |\cL|}{q} \right|
\le
q^{n/2} \sqrt{|\cP| \cdot |\cL|}.
\]
\end{thm}

We had initially sought to generalise Vinh's result to sets
\[
\{ (x,y) \in \bbF_q^2:
y = (x + a)^n + b \},
\]
but this had eluded us. We now achieve this. Our results apply to symmetric polynomials, i.e. $F(a,x) = F(x,a)$.

\begin{thm}
\label{MainThm}
Let $F(a,x) \in \bbF_q[a,x]$ be a non-diagonal, symmetric polynomial such that $\deg(F) < \mathrm{char}(\bbF_q)$. Let $\cP$ be a set of points and $\cQ$ a collection of subsets of $\bbF_q^2$ of the shape
\begin{equation}
\label{qab}
\fq_{a,b} =
\{ (x,y) \in \bbF_q^2:
F(a,x) + b + y = 0 \}
\end{equation}
in $\bbF_q^2$. Then
\[
\# \{ (\fp, \fq) \in \cP \times \cQ:
\fp \in \fq \}
= \frac{ |\cP| \cdot |\cQ| }q +
O_{\deg(F)}(q^{5/6} \sqrt{|\cP| \cdot |\cQ|}).
\]
\end{thm}

\subsection{Expanding polynomials in finite fields}

For a polynomial 
$P \in \bbF_q[x_1,\ldots,x_k]$
and sets $A_1, \ldots, A_k \subseteq \bbF_q$, we would usually expect
\[
P(A_1, \ldots, A_k)
:= \{ P(x_1, \ldots, x_k): x_1 \in A_1, \ldots, x_k \in A_k \}
\]
to be much larger than $A_1, \ldots, A_k$, i.e. $P$ should \emph{expand}. There are some algebraic obstructions, e.g. $F$ could be additively or multiplicatively structured, or some $A_i$ could be dense.
The \emph{Elekes--R\'onyai problem} \cite{dZ2018} is to demonstrate expansion away from such obstructions. The problem is also studied over more general fields, and is closely related to the sum--product problem.

In order to quantify the amount of expansion, Tao \cite{Tao} proposed the following hierarchy.
\begin{enumerate}[label=(\roman*)]
\item\emph{Weak asymmetric expansion}: there exist $c, C>0$ such that $$|P(A_1,\ldots,A_k)|\geq C^{-1}
\min\{|A_1|,\ldots,|A_k|\}^{1-c}q^c$$ whenever $|A_1|,\ldots,|A_k|\geq Cq^{1-c}$.
\item\emph{Moderate asymmetric expansion}: there exist $c, C>0$ such that $$|P(A_1,\ldots,A_k)|\geq C^{-1}q$$ whenever $|A_1|,\ldots,|A_k|\geq Cq^{1-c}$.
\item\emph{Almost strong asymmetric expansion}: there exist $c,C>0$ such that $$|P(A_1,\ldots,A_k)|\geq q-Cq^{1-c}$$ whenever $|A_1|,\ldots,|A_k|\geq Cq^{1-c}$.
\item\emph{Strong asymmetric expansion}: there exist $c,C>0$ such that $$|P(A_1,\ldots,A_k)|\geq q-C$$ whenever $|A_1|,\ldots,|A_k|\geq Cq^{1-c}$.
\item\emph{Very strong asymmetric expansion}: there exist $c,C>0$ such that $$P(A_1,\ldots,A_k)=\bbF_q$$ whenever $|A_1|,\ldots,|A_k|\geq Cq^{1-c}$.
\end{enumerate}
The qualifier ``asymmetric'' in the above refers to the fact that the sets $A_1,\ldots,A_k$ are allowed to differ. Typically $q$ is allowed to vary, and $P$ is allowed to vary in some family, meaning that $c,C$ do not depend on them.

In our previous paper, we established almost strong expansion for a Zariski-dense set of polynomials of degree $d$ in $d+1$ variables, strengthening and generalising the main result of \cite{PVZ2019}. We are now able to handle ternary polynomials of more general degree.

\begin{thm}
\label{MainExpansion}
Let $P\in\bbF_q[x,y,z]$ be a polynomial of the form
$$F(x,G(y,z))+H(y,z)+J(x),$$
where $F,G,H\in \bbF_q[x,y]$ and $J\in \bbF_q[x]$ are such that:
\begin{enumerate} 
[label=(\roman*)]
\item $F$ is non-diagonal and symmetric with $\deg(F)<\mathrm{char}(\bbF_q)$, and 
\item $G$ and $H$ are algebraically independent. 
\end{enumerate}
Then there exists a constant $C = C_{\deg(P)}$ such that if $X,Y,Z\subseteq\bbF_q$ and $|X|,|Y|,|Z|\geq Cq^{1/2}$ then
\[
|P(X,Y,Z)| = q-O_{\deg(P)}\left(\frac{q^{11/3}}{|X|\cdot|Y|\cdot|Z|}\right)\text{.}
\]
\end{thm}

\begin{cor}
Let $\eps > 0 $, let $P$ be a polynomial of the form given in Theorem~\ref{MainExpansion}, and let $X,Y,Z\subseteq\bbF_q$ be such that $|X|,|Y|,|Z|\gg q^{8/9+\varepsilon}$. Then
$$|P(X,Y,Z)|=q-O_{\varepsilon,\deg(P)}(q^{1-3\varepsilon})\text{.}$$
\end{cor}

\begin{cor}
\label{erdcor}
Let $\eps > 0$, let $P(x,y,z)=(x-y)^k+z$, where $k\geq2$ is an integer, and let $q$ be a power of a prime $p > k$. Let $X,Y,Z\subseteq\bbF_q$ be such that $|X|,|Y|,|Z|\gg q^{8/9+\varepsilon}$. Then
$$|P(X,Y,Z)|=q-O_{\varepsilon,k}(q^{1-3\varepsilon})\text{.}$$
\end{cor}

Aside from the aforementioned results, Shkredov \cite{Shkredov} established very strong expansion for $x^2 + xy + z$. As for higher-degree ternary polynomials, we were previously successful at demonstrating almost strong expansion \cite[Theorem 1.6]{AC} only in the special case
\[
P(x,y,z) = xG(y,z) + H(y,z) + ax^k,
\]
where $G$ and $H$ are algebraically independent. Theorem \ref{MainExpansion} is more general. For example, it handles
\[
P(x,y,z) = (x-y)^k + z,
\]
whereas our previous approach does not reach such polynomials.

\subsection{Methods}

We will see that Theorem \ref{MainExpansion} is easy to deduce from Theorem \ref{MainThm}. To prove Theorem \ref{MainThm}, we use Vinh's spectral framework \cite{Vinh1}, which begins by realising incidences as edges of a regular graph. The main task is to bound the second largest eigenvalue, in absolute value. Vinh is able to use the square of an adjacency matrix to swiftly achieve this in the case of point--line incidences. In our higher-degree setting, we need to cube the matrix, giving rise to a family of algebraic curves on which we need to count $\bbF_q$ points. The Weil bound is suitable for this purpose, as we are able to show that almost all curves in the family are absolutely irreducible.

\subsection*{Organisation}
We will prove Theorem \ref{MainThm} in \S \ref{IncidenceSection}. Then, in \S \ref{applications}, we will deduce Theorem \ref{MainExpansion}.

\subsection*{Notation}

We use the Vinogradov notations $\ll$ and $\gg$, as well as the Bachmann--Landau notation $O(\cdot)$.
Let $f$ and $g$ be complex-valued functions. We write $f \ll g$ or $f=O(g)$ if $|f|\leq C|g|$ pointwise, for some constant $C>0$. We use a subscript within these notations to indicate possible dependence for the implied constant $C$. 

\subsection*{Funding}

NA was funded through the Engineering
and Physical Sciences Research Council Doctoral Training Partnership at the
University of Warwick.

\subsection*{Rights}

For the purpose of open access, the authors have applied a Creative Commons Attribution (CC-BY) licence to any Author Accepted Manuscript version arising from this submission.

\section{An incidence bound}
\label{IncidenceSection}

In this section, we prove Theorem \ref{MainThm}. The first subsection will establish the result assuming a lemma about absolute irreducibility. The latter will then be demonstrated in the second subsection.

\subsection{Spectral theory}
\begin{defn}
For a symmetric polynomial $F\in\bbF_q[x,y]$, we define the graph $\cG_F$ as the graph with set of vertices $\bbF_q^2$ and an edge between $(a,b)$ and $(x,y)$ if and only if
$$F(a,x)+b+y=0\text{.}$$
\end{defn}
We will prove the following spectral bound for the graphs $\cG_F$.

\begin{lemma}
\label{boundl2}
Let $F\in\bbF_q[x,y]$ be a non-diagonal, symmetric polynomial such that $\deg(F) < \mathrm{char}(\bbF_q)$. Then $\lambda_2(\cG_F)\ll_{\deg(F)}q^{5/6}$.
\end{lemma}

\begin{proof}
[Proof of Lemma \ref{boundl2} assuming Lemma \ref{zc}]
Let $A$ be the adjacency matrix of $\cG=\cG_F$, the rows and columns of which are indexed by points in $\bbF_q^2$. We consider the matrix $A^3$. 

The entry of $A^3$ in the row corresponding to $(a,b)$ and the column corresponding to $(c,d)$ is the number of paths of length $3$ in $\cG$ connecting $(a,b)$ to $(c,d)$. This is the number of solutions $(x_1,x_2,y_1,y_2)\in\bbF_q^4$ to the system of equations
$$
\begin{cases}
    F(a, x_1) + b + y_1 = 0 \\
    F(x_1, x_2) + y_1 + y_2 = 0 \\
    F(x_2, c) + y_2 + d = 0\text{.}
\end{cases}
$$
Solutions to this system are in bijection with pairs $(x_1,x_2)\in\bbF_q^2$ satisfying
\begin{equation}
\label{curvedef}F(a,x_1)+F(x_2,c)-F(x_1,x_2)+b+d=0\text{,}
\end{equation}
with mutually inverse bijections being given by
$$(x_1,x_2,y_1,y_2)\mapsto(x_1,x_2)$$
and
$$(x_1,x_2)\mapsto(x_1,x_2,-F(a,x_1)-b,-F(x_2,c)-d)\text{.}$$
Let $C_{\mathbf{a}}=C/\bbF_q$ be the affine plane curve defined in coordinates $(x_1,x_2)$ by \eqref{curvedef}, where $\mathbf{a}=(a,b,c,d)$, so that, as argued above, $$A^3_{(a,b),(c,d)}=\#C(\bbF_q)\text{.}$$ 
Observe that \eqref{curvedef} indeed defines a curve; the polynomial on the left-hand side of \eqref{curvedef} does not vanish, since $F$ is assumed not to be diagonal. 

By the Weil bound, 
$$\#C(\bbF_q)=\begin{cases}
q+O(\sqrt{q})&\text{ if $C$ is absolutely irreducible}\\
O(q)&\text{ otherwise.}
\end{cases}
$$
Denote by $U$ the $q^2\times q^2$ matrix all the entries of which equal $1$. The upshot is that
$$A^3=qU+M \text{,}$$
where $M$ is a $q^2\times q^2$ real matrix, the entries of which are indexed by points $\mathbf{a}\in\bbF^4$, such that
\[
M_{\mathbf{a}} = \begin{cases}
O(\sqrt{q}) &\text{ if $C_\mathbf{a}$ is absolutely irreducible} \\
O(q) &\text{ otherwise.}
\end{cases}
\]
Lemma \ref{zc}, which we will prove in \S \ref{absirr}, asserts that the points $\mathbf{a}$ such that $C_\mathbf{a}$ is not absolutely irreducible are contained in a proper subvariety $Z \subset \bbA^4$.

We have $\lambda_1(\cG)=q$, since $\cG$ is clearly $q$-regular. Note that the eigenspace corresponding to the eigenvalue $1$ is spanned by the vector $\mathbf{u}=(1,\ldots,1)$. Since $A$ is symmetric, the spectral theorem implies that
$$|\lambda_2(\cG)|^3=\max\{\lambda\in\bR:\|A^3\mathbf{x}\|_2\leq\lambda\|\mathbf{x}\|_2\text{ for any }\mathbf{x}\text{ with }\mathbf{x}\cdot\mathbf{u}=0\}\text{.}$$
For $\mathbf{x}$ with $\mathbf{x}\cdot\mathbf{u}=0$, we have $U\mathbf{x} = \bzero$ and hence
\[
A^3\mathbf{x}=(qU+M)\mathbf{x}=M\mathbf{x}\text{.}
\]
Now
$$\|A^3\mathbf{x}\|_2=\|M\mathbf{x}\|_2\leq\|M\|_2\|\mathbf{x}\|_2\text{,}$$
by Cauchy-Schwarz, where
\[
\|M\|_2 = \sqrt{\sum_{\mathbf{a}\in\bbF_q^4} M_\mathbf{a}^2}\text{.}
\]
It follows that
\begin{equation}
\label{bdl2}|\lambda_2(\cG)|\leq\left(\sum_{\mathbf{a}\in\bbF_q^4}M_\mathbf{a}^2\right)^{1/6}\text{.}
\end{equation}

Using that $\#Z(\bbF_q)=O(q^3)$, since $Z$ is a proper subvariety of $\bbA^4$,
\begin{align*}\sum_{\mathbf{a}\in\bbF_q^4}M_\mathbf{a}^2&=\sum_{\mathbf{a}\in Z(\bbF_q)}M_\mathbf{a}^2+\sum_{\mathbf{a}\notin Z(\bbF_q)}M_\mathbf{a}^2\\
&=O(q^3\cdot q^2)+O(q^4\cdot q)=O(q^5)\text{.}
\end{align*}
Inserting this into \eqref{bdl2} completes the proof.
\end{proof}

\begin{proof}[Proof of Theorem \ref{MainThm} assuming Lemma \ref{zc}]
The result follows immediately by applying the expander mixing lemma \cite[Corollary 9.2.5]{AlonSpencer} to the graph $\cG_F$, together with Lemma \ref{boundl2}.
\end{proof}

\subsection{Absolute irreducibility}
\label{absirr}

In this subsection, we will establish the following lemma used in the proof of Lemma \ref{boundl2}. 

\begin{lemma}
\label{zc}
Let $F\in\bbF_q[x,y]$ be a non-diagonal, symmetric polynomial such that $\deg(F) < \mathrm{char}(\bbF_q)$.
Then there exists a proper subvariety $Z\subset \bbA^4$ of degree $O_{\deg(F)}(1)$ such that, for $\mathbf{a}=(a,b,c,d)\notin Z(\bbF_q)$, the curve $C_\mathbf{a}$ defined by \eqref{curvedef} is absolutely irreducible.
\end{lemma}

To this end, we will require the following auxiliary result, the proof of which we defer to Appendix \ref{polcomp}.

\begin{lemma}
\label{polcomp1}
Let $k$ be a field and let $P,Q,R,S\in k[x]$ be polynomials such that $0 < \deg(P)=\deg(R)<\mathrm{char}(k)$ and the compositions $P\circ Q$ and $R\circ S$ are equal. Then there exists a linear polynomial $L \ne 0$ such that $S=L\circ Q$.
\end{lemma}

We will also require the following characterisation of additive polynomials \cite[Chapter 1]{Gos1998}.

\begin{lemma}
\label{polcomp3}
Let $k$ be an infinite field of finite characteristic $p$, and let $P\in k[X]$ satisfy
\[
P(x+y) = P(x) + P(y)
\qquad (x, y \in k).
\]
Then 
\[
P = \sum_{j=0}^K
a_j X^{p^j}
\]
for some $a_0, \ldots, a_K \in k$.
\end{lemma}

\begin{proof}
[Proof of Lemma \ref{zc}]
We begin by observing that reducibility of $C_\mathbf{a}$ can be characterised in terms of polynomial equations in the coefficients of the polynomial defining $C_{\mathbf{a}}$, see \cite[Chapter V, Theorem 2A]{Sch1976}. Thus, there exists a subvariety $Z \subset \bbA^4$ of degree $O_{\deg(F)}(1)$ such that $C_\mathbf{a}$ is absolutely irreducible if and only if $\mathbf{a}\notin Z$. It remains to prove that $Z$ is proper, in other words, that $C_\mathbf{a}$ is irreducible over $\overline{\bbF_q}$ for some $\mathbf{a} \in \overline{\bbF_q}^4$. Henceforth, we assume for a contradiction that $C_\mathbf{a}$ is reducible over $\overline{\bbF_q}$ for every $\mathbf{a} \in \overline{\bbF_q}^4$.

Our assumption implies that the polynomial
$$F(a,x_1)+F(x_2,c)-F(x_1,x_2)+t$$
is reducible for any $a,c,t\in\overline{\bbF_q}$. By \cite[\S3.3, Corollary 1]{Sch2000}, this forces the polynomial $F(a,x_1)+F(x_2,c)-F(x_1,x_2)$ to be composite --- namely it has the form $P(Q(x_1,x_2))$ with $\deg(P)\geq2$ --- for any $a,c\in\overline{\bbF_q}$. Specialising $a=0$, we deduce that the polynomial 
\[
F(0,x_1)+F(x_2,c)-F(x_1,x_2) \in \bbF_q(c)[x_1,x_2]
\]
has the form $P(Q(x_1,x_2))$, for some polynomials $P$ and $Q$ with coefficients in $\overline{\bbF_q(c)}$ and $\deg(P)\geq 2$. 

Note that the coefficients of $P$ and $Q$ lie in some finite extension of $\bbF_q(c)$, say
\[
K = \bbF_q(c, t_1, \ldots, t_r),
\]
where $P_i(c, t_i) = 0$ for some $P_1, \ldots, P_r \in \bbF_q[x,y]$. These define a variety
\[
X = \{ P_i(c, t_i) = 0 \} \subset \bbA^{r+1}_{c,t_1,\ldots,t_r}
\]
over $\bbF_q$, which has to be a curve, since its function field $K$ has transcendence degree $1$.
This gives rise to a polynomial identity
\begin{equation}
\label{compid}
F(0,x_1)+F(x_2,\pi(p))-F(x_1,x_2)=P_p(Q_p(x_1,x_2)) \text{,}
\end{equation}
where $\pi:X\to\bbA^1$ is a non-constant morphism, and $P_p,Q_p$ are polynomials whose coefficients are algebraically parametrised by a point $p\in X$.

For $p,p'\in X$ with 
\[
F(x, \pi(p)) \ne 
F(x, \pi(p'))
\in \overline{\bbF_q}[x],
\]
the above yields
\begin{equation}
\label{difference}
F(x_2,\pi(p))-F(x_2,\pi(p'))=P_p(Q_p(x_1,x_2))-P_{p'}(Q_{p'}(x_1,x_2))\text{.}
\end{equation}
If we fix $p,p'$ and regard the above as an identity of polynomials in the variable $x_1$ with coefficients in $\overline{\bbF_q}(x_2)$, then we have the equality of polynomial compositions
\[
(P_{p'}+F(x_2,\pi(p))-F(x_2,\pi(p')))\circ Q_{p'}(\cdot,x_2)=P_p\circ Q_p(\cdot,x_2)\text{.}
\]

For generic $p,p'\in X$, the degrees in $x_1$ of 
\[
P_p, \qquad
P_{p'}+F(x_2,\pi(p))-F(x_2,\pi(p'))
\]
are equal. Now Lemma \ref{polcomp1} implies that for generic --- and hence for all --- $p,p'\in X$ the polynomials $Q_{p'}(\cdot,x_2)$ and $Q_p(\cdot,x_2)$ are related by left-composition with a linear polynomial with coefficients in $\overline{\bbF_q}(x_2)$. In other words, the family of polynomials $Q_p(x_1,x_2)$ has the form
\[
Q_p(x_1,x_2)=f_p(x_2)h(x_1,x_2)+g_p(x_2) \text{,}
\]
for some families of polynomials $f_p,g_p\in \overline{\bbF_q}[x_2]$ with $f_p \ne 0$, and some $h\in\bbF_q[x_1,x_2]$. We assume as we may that the polynomials $f_p$ do not have a common factor, for such a factor can be absorbed into $h$.

Inserting the above into \eqref{difference} yields
\begin{align*}
&F(x_2,\pi(p))-F(x_2,\pi(p'))\\
&=P_p(f_p(x_2)h(x_1,x_2)+g_p(x_2))-P_{p'}(f_{p'}(x_2)h(x_1,x_2)+g_{p'}(x_2))\text{.}
\end{align*}
Since $h$ is non-constant --- as otherwise the left-hand side of \eqref{compid} would not depend on $x_1$, which is clearly a contradiction --- we obtain the polynomial identity
\[
F(x_2,\pi(p))-F(x_2,\pi(p'))=P_p(f_p(x_2)z+g_p(x_2))-P_{p'}(f_{p'}(x_2)z+g_{p'}(x_2))\text{.}
\]
Let $m\geq2$ be the generic degree of $P_p$, and let $a_p$ be its leading coefficient. Comparing leading coefficients in $z$ yields
\[
a_pf_p(x_2)^m=a_{p'}f_{p'}(x_2)^m\text{,}
\]
which implies that any two $f_p$ and $f_{p'}$ are related by multiplication by a root of unity. Since the $f_p$ were assumed not to share a common factor, this implies that $(f_p)_p$ is a family of constant polynomials. We thus obtain
\[
P_p\left(f_pz+g_p(x_2)\right)-P_{p'}(f_{p'}z+g_{p'}(x_2))=F(x_2,\pi(p))-F(x_2,\pi(p'))\text{.}
\]

We now make the variable change $w=f_{p'}z+g_{p'}(x_2)$ and observe that
$$P_p \left(
\frac{f_p}{f_{p'}}(w-g_{p'}(x_2))
\right) = P_{p'}(w)+F(x_2,\pi(p))-F(x_2,\pi(p'))\text{.}$$
The right-hand side is a polynomial in $x_2$ and $w$ where no monomial involving both $x_2$ and $w$ shows up with a non-zero coefficient. On the other hand, the left-hand side is of the form
\[
P_p(aw + \ell(x_2)) \text{,}
\]
where $\ell$ is non-constant.
This expands as 
\[
P_p(aw) + P_p(\ell(x_2)) + C + M(w,x_2)
\]
for some constant $C$, where $M(w,x_2)$ is a linear combination of monomials involving both $w$ and $x_2$. We must have $M(w,x_2)=0$, so $P_p$ satisfies the identity 
\[
P_p(x+y)=P_p(x)+P_p(y)+C
\text{.} 
\]

The upshot is that the polynomial $P_p+C$ satisfies the assumption of Lemma \ref{polcomp3}. Therefore, either $\deg(P_p) = 1$ or $\deg(P_p) \ge \mathrm{char}(\bbF_q)$. The latter would contradict our assumption on the degree of $F$, in light of \eqref{compid}. Therefore $P_p$ has degree $1$, which contradicts our assumption that it has degree at least $2$.
\end{proof}

\section{Expanding polynomials}
\label{applications}

In this section, we prove Theorem \ref{MainExpansion}.
We begin by showing that
$$\#\{(G(y,z),H(y,z)):(y,z)\in Y\times Z\}\gg |Y|\cdot|Z|\text{.}$$
To this end, we argue as follows. Identify $\bbF_q^2$ with affine space $\bbA_{\bbF_q}^2$. The condition that $G$ and $H$ are algebraically independent implies that the map $\varphi$ given by $\varphi(y,z)=(G(y,z),H(y,z))$ is dominant, i.e. its image is not contained in a proper subvariety of $\bbA_{\bbF_q}^2$. Therefore, if we denote by $k=O(1)$ the degree of $\varphi$, there exists a proper subvariety $\cX \subseteq \bbA_{\bbF_q}^2$, of degree $O(1)$ by \cite[Lemma 3.1]{AC}, such that $\varphi^{-1}(\mathbf{a})$ is finite whenever $\mathbf{a} \notin \cX$. Since $\varphi$ is at most $k$-to-$1$ outside $\varphi^{-1}(\cX)$, it follows that
\begin{align*}
|\varphi(Y\times Z)|&\geq|\varphi(\bbA_{\bbF_q}^2\setminus\varphi^{-1}(\cX))(\bbF_q)\cap (Y\times Z))|\\
&\geq\frac{1}{k}|(\bbA_{\bbF_q}^2\setminus\varphi^{-1}(\cX))(\bbF_q)\cap (Y\times Z)|\\
&\geq\frac{1}{k}(|Y|\cdot|Z|-|\varphi^{-1}(\cX)(\bbF_q)|)\\
&=\frac{1}{k}(|Y|\cdot|Z| + O(q))\text{.}
\end{align*}
In the final line we used that $\varphi^{-1}(\cX)$, being a proper subvariety of $\bbA_{\bbF_q}^2$, has dimension at most $1$, and its degree is bounded in terms of the degrees of $\cX$ and $\varphi$. It is now clear that, for sufficiently large $C$, if $|Y|,|Z|\geq Cq^{1/2}$, then the above is at least $c|Y|\cdot|Z|$ for a suitably small constant $c > 0$.

Recall \eqref{qab}, and consider the set
$$\cQ=\{\fq_{G(y,z),H(y,z)}:(y,z)\in Y\times Z\}\text{.}$$
We just proved that $|\cQ|\gg|Y|\cdot|Z|$. Define also $W=\bbF_q\setminus P(X,Y,Z)$, and set
$$\cP = \{(x, J(x)-w): x \in X, w \in W \}
\text{,}$$
noting that $|\cP|=|X|\cdot|W|$. It follows from the definitions of $\cP$, $\cQ$, $P$ and $W$ that if $\fp\in\cP$ and $\fq\in\cQ$ then $\fp\notin\fq$. Therefore Theorem \ref{MainThm} yields
$$\frac{|\cP|\cdot|\cQ|}{q}=O(q^{5/6}\sqrt{|\cP|\cdot|\cQ|})\text{,}$$
whence
$$|\cP|\cdot|\cQ|\ll q^{11/3}\text{.}$$
On using that 
\[
|\cQ|\gg |Y|\cdot|Z|, \quad
|\cP|=|X|\cdot|W|, 
\quad
|W|=q-|P(X,Y,Z)|,
\]
the result follows.

\appendix

\section{On polynomial composition}
\label{polcomp}
Here we prove Lemma \ref{polcomp1}. A variant of this featured as Problem 5 at the 14th edition of the Romanian Masters of Mathematics, and the proof that we present here is adapted from the official solutions to that problem.

\begin{proof}[Proof of Lemma \ref{polcomp1}]
Since $P\circ Q=R\circ S$, one has 
\[
\deg(P)\deg(Q)=\deg(R)\deg(S)
\]
which, since $\deg(P)=\deg(R)$, implies that $\deg(Q)=\deg(S)$. Let $d$ be this common degree. If $\alpha$ and $\beta$ denote the leading coefficients of $Q$ and $S$, respectively, and $c=\alpha\beta^{-1}$, then $S-cQ$ is a polynomial of degree smaller than $d$. If this polynomial is constant, we are done. So we assume henceforth that it is non-constant, i.e. that $S=cQ+M$ where $M$ is a polynomial with $0 < \deg(M) < d$.

We then have
$$R(cQ(x)+M(x))=P(Q(x))\text{,}$$
and the left-hand side expands as a linear combination of products of powers of $Q(x)$ and $M(x)$. Among the terms that contain positive powers of $M(x)$, the one with the largest degree is $\gamma rc^{r-1}Q(x)^{r-1}M(x)$, where $r$ is the common degree of $R$ and $P$ and $\gamma$ is the leading coefficient of $R$. Recall that $0 < r < \mathrm{char}(k)$. Therefore $rc^{r-1}Q^{r-1}M$ differs  from $T\circ Q$ by a polynomial of lower degree, for some polynomial $T$, and in particular 
$$\deg(Q^{r-1}M)=\deg(T\circ Q)\text{.}$$

The degree of the right-hand side is divisible by $d$, whereas the degree of the left-hand side is $d(r-1)+\deg(M)$, which is not divisible by $d$ because $0<\deg(M)<d$. This contradiction completes the proof of the lemma.
\end{proof}

\providecommand{\bysame}{\leavevmode\hbox to3em{\hrulefill}\thinspace}

\end{document}